\documentclass[11pt,a4paper]{amsart}
\usepackage{
amssymb,
amsmath}

\usepackage[bookmarks,colorlinks,pagebackref]{hyperref}

\numberwithin{equation}{section}
\newtheorem{theorem}{Theorem}[section]
\newtheorem{lemma}[theorem]{Lemma}
\newtheorem{proposition}[theorem]{Proposition}
\newtheorem{corollary}[theorem]{Corollary}

\theoremstyle{definition}
\newtheorem{definition}[theorem]{Definition}

\newtheorem{problem}[theorem]{Problem}
\newtheorem{remark}[theorem]{Remark}

\newtheorem{notation}[theorem]{Notation}

\newcommand\Supp{\operatorname{Supp}}
\newcommand\Proj{\operatorname{Proj}}
\newcommand\Ass{\operatorname{Ass}}

\newcommand\Ann{\operatorname{Ann}}

\newcommand\Hom{\operatorname{Hom}}

\newcommand\Rad{\operatorname{Rad}}

\newcommand{\au}{\underline a}

\newcommand{\Cech}{{\Check {C}}_{\au}}

\author[Khadam]{M. Azeem Khadam}
\address{Abdus Salam School of Mathematical Sciences, GCU, Lahore Pakistan}
\email{azeemkhadam@gmail.com}

\author[Schenzel]{Peter Schenzel}
\address{Martin-Luther-Universit\"at Halle-Wittenberg,
Institut f\"ur Informatik, D --- 06 099 Halle (Saale), Germany}
\email{peter.schenzel@informatik.uni-halle.de}
\thanks{The first named author is grateful to DAAD for the support this research under 
	grant number 91524811}

\title[Local cohomology]{About a variation of local cohomology}

\begin{document}

\begin{abstract} Let $\mathfrak{q}$ denote an ideal of a local ring $(A,\mathfrak{m})$. 
	For a system of elements $\au = a_1,\ldots,a_t$ such that $a_i \in \mathfrak{q}^{c_i}, 
	i = 1, \ldots,t,$ and $n \in \mathbb{Z}$ we investigate a subcomplex resp. a factor complex of the \v{C}ech 
	complex $\check{C}_{\au} \otimes_A M$ for a finitely generated $A$-module $M$. We start with 
	the inspection of these cohomology modules that approximate in a certain sense the local cohomology 
	modules $H^i_{\au}(M)$ for all $i \in \mathbb{N}$. In the case of an $\mathfrak{m}$-primary 
	ideal $\au A$ we prove the Artinianness of these cohomology modules and characterize the 
	last non-vanishing among them.
\end{abstract}


\maketitle

\section{Introduction} 
Let $(A,\mathfrak{m},\Bbbk)$ denote a local ring. Let $\mathfrak{q} \subset A$ be an ideal 
and $\au = a_1,\ldots,a_t$ denote a system of elements of $A$ such that $a_i \in \mathfrak{q}^{c_i}, 
i = 1,\ldots,t$. For a finitely generated $A$-module $M$ and an integer $n \in \mathbb{N}$ we 
define a complex $K_{\bullet}(\au,\mathfrak{q},M;n)$ as the subcomplex of the Koszul complex 
$K_{\bullet}(\au;M)$, where $K_i(\au,\mathfrak{q},M;n) = \oplus_{1\leq j_1 < \ldots < j_i \leq t} 
\mathfrak{q}^{n-c_{j_1}- \ldots- c_{j_i}}M$ is included in $K_i(\au;M)$ and the boundary map 
is defined as the restriction of the maps in the Koszul complex. In other words, 
$K_{\bullet}(\au,\mathfrak{q},M;n)$ is the $n$-th graded component of the Koszul complex 
$K_{\bullet}(\underline{aT^c};R_M(\mathfrak{q}))$. Here $R_M(\mathfrak{q})$ denotes the Rees 
module of $M$ with respect to $\mathfrak{q}$ and $\underline{aT^c}= a_1T^{c_1},\ldots,a_tT^{c_t}$, 
where $a_iT^{c_i}, i = 1,\ldots,t,$ is the element $a_iT^{c_i} \in R_A(\mathfrak{q})$, the  associated 
element of degree $c_i$ in the Rees ring $R_A(\mathfrak{q}) = \oplus_{n \geq 0} \mathfrak{q}^n T^n$ 
(see Section 3 for more details). We denote this complex by $K_{\bullet}(\au,\mathfrak{q},M;n)$. 
The cokernel of this embedding provides a complex $\mathcal{L}_{\bullet}(\au,\mathfrak{q},M;n)$. 
In the case that $\au A, \mathfrak{q}$ are $\mathfrak{m}$-primary ideals and $t = \dim M$ 
the Euler characteristic of $\mathcal{L}_{\bullet}(\au,\mathfrak{q},M;n)$ gives for all $n \gg 0$ 
the value $c_1\cdots c_t \, e_0(\mathfrak{q};A)$, where $e_0(\mathfrak{q};A)$ denotes the Hilbert-Samuel 
multiplicity of $\mathfrak{q}$ (see Section 5 for the details).  

A similar construction based on the co-Koszul complex provides complexes $K^{\bullet}(\au,\mathfrak{q},M;n)$ 
and $\mathcal{L}^{\bullet}(\au,\mathfrak{q},M;n)$ (see Section 4). It follows that 
$\{K^{\bullet}(\au^k,\mathfrak{q},M;n)\}_{k \geq 1}$ with $\au^k = a_1^k,\ldots,a_t^k,$ forms a 
direct system of complexes and its direct limit $\check{C}(\au,\mathfrak{q},M;n)$ is a subcomplex of the 
\v{C}ech complex $\check{C}_{\au} \otimes_A M$ with the factor complex $\check{\mathcal{L}}^{\bullet}(\au,\mathfrak{q},M;n)$
and their cohomology modules $\check{H}^i(\au,\mathfrak{q},M;n)$ and $\check{L}^i(\au,\mathfrak{q},M;n)$ resp. 
Therefore there is a long exact cohomology sequence 
\[
\ldots \to \check{H}^i(\au,\mathfrak{q},M;n) \to H^i_{\au}(M) \to \check{L}^i(\au,\mathfrak{q},M;n) \to \ldots.
\]
Since $\check{H}^i(\au,\mathfrak{q},M;n) = 0$ for $i > t$, there is an epimorphism $ H^t_{\au}(M) \to 
\check{L}^t(\au,\mathfrak{q},M;n)$. So the non-vanishing of $\check{L}^t(\au,\mathfrak{q},M;n)$ for 
some $n \in \mathbb{N}$ is an obstruction for the vanishing of $H^t_{\au}(M)$. 
The main results of the present manuscript is a contribution to the study of 
$\check{H}^i(\au,\mathfrak{q},M;n)$ and $\check{L}^i(\au,\mathfrak{q},M;n)$. In particular we prove the following 
results:

\begin{theorem} \label{intro}
	With the previous notation suppose that $\au A$ is an $\mathfrak{m}$-primary ideal. 
	\begin{itemize}
		\item[(a)] $\check{H}^i(\au,\mathfrak{q},M;n) \; \text{ and } \; \check{L}^i(\au,\mathfrak{q},M;n)$
		are Artinian $A$-modules for all $i,n \in \mathbb{N}$.
		\item[(b)] $\check{H}^t(\au,\mathfrak{q},M;n) = 0$ for all $n \gg 0$.
		\item[(c)] $\check{L}^t(\au,\mathfrak{q},M;n) \not= 0$ for all $n \gg 0$.
	\end{itemize}
\end{theorem} 

For the proof of (a) we refer to \ref{seq-4}. The claim of (b) is a particular case of \ref{van-2} and 
(c) is shown in \ref{nonv-5}. Besides of these results there are statements about the structure of the 
generalized Koszul homology and co-homology modules and their Euler characteristics. For induction 
arguments we provide some exact sequences (see Section 7). Of a particular interest is a generalization of the 
notion of superficial sequences and the condition ($\star$) used in \ref{van-2}. A further 
investigation about the vanishing and the rigidity of the generalized Koszul and co-Koszul complexes 
is in preparation. 

In the special case of $\mathfrak{q} = \au A$ for a system of elements $\au = a_1,\ldots, a_t$ of $A$ 
we have the following:

\begin{corollary} \label{intro-1}
	Let $\au = a_1,\ldots,a_t$ denote a system of elements of $A$. For a finitely generated $A$-module $M$ 
	there are isomorphisms 
	\[
	H^i_{\au}(M) \cong \check{L}^i(\au,\au,M;n)
	\]
	for all $n \gg 0$ and all $i \geq0$
\end{corollary}

This follows by view of \ref{cech-4}. That is, in a certain sense the cohomology $\check{L}^i(\au,\mathfrak{q},M;n)$ 
provides some additional structure on the usual local cohomology modules. 

As a source for basic notions in Commutative Algebra we refer to \cite{AM} or \cite{mR}. For results 
on homological algebra we refer to \cite{jR} and \cite{cW}. The local cohomology is developed in 
\cite{aG} (see also \cite{pS0}). For a system of elements $\au$ we write $H^i_{\au}(\cdot), i \in \mathbb{N},$ 
for the local cohomology modules with respect to the ideal generated by $\au$ (see \cite{aG}). 

\section{Preliminaries}
First let us fix the notations we will use in the following. For the basics on $\mathbb{N}$-graded 
structures we refer e.g. to \cite{GW}.

\begin{notation} \label{not-1}
	(A) We denote by $A$ a commutative Noetherian ring with $0 \not= 1$. 
	For an ideal we write $\mathfrak{q} \subset A$. An $A$-module is denoted by $M$. 
	Mostly we consider  $M$ as finitely generated. \\
	(B) We consider the Rees and form rings of $A$ with respect 
	to $\mathfrak{q}$ by 
	\[
	R_A(\mathfrak{q}) = \oplus_{n \geq 0} \mathfrak{q}^n \,T^n \subseteq A[T] \,
	\text{ and }\, G_A(\mathfrak{q}) = \oplus_{n \geq 0} \mathfrak{q}^n/\mathfrak{q}^{n+1}.
	\]
	Here $T$ denotes an indeterminate over $A$. Both rings are naturally $\mathbb{N}$-graded. For an $A$-module $M$ we define 
	the Rees and form modules in the corresponding way by 
	\[
	R_M(\mathfrak{q}) = \oplus_{n \geq 0} \mathfrak{q}^n M \,T^n \subseteq M[T] \,
	\text{ and } \, G_M(\mathfrak{q}) = \oplus_{n \geq 0} \mathfrak{q}^nM/\mathfrak{q}^{n+1}M.
	\]
	Note that $R_M(\mathfrak{q})$ is a graded $R_A(\mathfrak{q})$-module and 
	$G_M(\mathfrak{q})$ is a graded $G_A(\mathfrak{q})$-module. Note that $R_A(\mathfrak{q})$ 
	and $G_A(\mathfrak{q})$ are both Noetherian rings. In case $M$ is a finitely generated 
	$A$-module then $R_M(\mathfrak{q})$ resp. $G_M(\mathfrak{q})$ is finitely generated 
	over $R_A(\mathfrak{q})$ resp. $G_A(\mathfrak{q})$. \\
	(C) There are the following two short exact sequences of graded modules
	\begin{gather*}
	0 \to R_M(\mathfrak{q})_{+}[1] \to R_M(\mathfrak{q}) \to G_M(\mathfrak{q}) \to 0 \text{ and }\\
	0 \to R_M(\mathfrak{q})_{+} \to R_M(\mathfrak{q}) \to M \to 0,
	\end{gather*}
	where $R_M(\mathfrak{q})_{+} = \oplus_{n > 0} \mathfrak{q}^n M \,T^n$. \\
	(D) Let $m \in M$ and $m \in \mathfrak{q}^c M \setminus \mathfrak{q}^{c+1} M$. Then we define  
	$m^{\star} := m + \mathfrak{q}^{c+1} M \in [G_M(\mathfrak{q})]_c$. If $m \in \cap_{n \geq 1} \mathfrak{q}^n M$, 
	then we write $m^{\star} = 0$. $m^{\star}$ is called the initial element of $m$ in $G_M(\mathfrak{q})$ and $c$ is called the initial degree of $m$.
	Here $[X]_n, n \in \mathbb{Z},$ denotes the $n$-th graded component of an $\mathbb{N}$-graded module $X$.
\end{notation}

For these and related results we refer to \cite{GW} and \cite{SH}. Another feature for the 
investigations will be the use of Koszul complexes. 

\begin{remark} \label{kos-1} ({\sl Koszul complex.})
	(A) Let $\underline{a} = a_1,\ldots,a_t$ denote a system of elements of the ring $A$. The Koszul complex 
	$K_{\bullet}(\underline{a};A)$ is defined as follows: Let $F$ denote a free $A$-module with basis $e_1,\ldots,e_t$. 
	Then $K_i(\underline{a};A) = \bigwedge^i F$ for $i = 1,\ldots,t$. A basis of $K_i(\underline{a};A)$ is given by the wedge 
	products $e_{j_1} \wedge \ldots \wedge e_{j_i}$ for $1 \leq j_1 < \ldots < j_i \leq t$. The boundary 
	homomorphism $K_i(\underline{a};A) \to K_{i-1}(\au;A)$ is defined by 
	\[
	d_{j_1 \ldots j_i} :
	e_{j_1} \wedge \ldots \wedge e_{j_i} \mapsto \sum_{k=1}^{i} (-1)^{k+1} a_{j_k} e_{j_1}\wedge \ldots \wedge \widehat{e_{j_k}} 
	\wedge \ldots \wedge e_{j_i}
	\] 
	on the free generators $e_{j_1} \wedge \ldots \wedge e_{j_i}$. \\
	(B) Another way of the construction of $K_{\bullet}(\au;A)$ is inductively by the mapping cone. 
	To this end let $X$ denote a complex of $A$-modules. Let $a \in A$ denote an element of $A$. The 
	multiplication by $a$ on each $A$-module $X_i, i \in \mathbb{Z},$ induces a morphism of complexes 
	$m_a : X \to X$. We define $K_{\bullet}(a;X)$ as the mapping cone $\operatorname{Mc} (m_a)$. Then we define inductively 
	\[
	K_{\bullet}(a_1,\ldots,a_t;A) = K_{\bullet}(a_t,K_{\bullet}(a_1,\ldots,a_{t-1};A)).
	\]
	It is easily seen that 
	\[
	K_{\bullet}(\au;A) \cong K_{\bullet}(a_1;A) \otimes_A \cdots \otimes_A K_{\bullet}(a_t;A).
	\]
	Therefore it follows that $K_{\bullet}(\au;A) \cong K_{\bullet}(\au_{\sigma};A)$, where 
	$\au_{\sigma} = a_{\sigma(1)}, \ldots, a_{\sigma(t)}$ with a permutation $\sigma$ on $t$ letters. 
	For an $A$-complex $X$ we define $K_{\bullet}(\au;X) = K_{\bullet}(\au;A)\otimes_A X$. We write 
	$H_i(\au;X), i \in \mathbb{Z},$ for the $i$-th homology of $K_{\bullet}(\au;X)$. A short exact 
	sequence of $A$-complexes $0 \to X' \to X \to X'' \to 0$ induces a long exact homology sequence 
	for the Koszul homology
	\[
	\ldots \to H_i(\au;X') \to H_i(\au;X) \to H_i(\au;X'') \to H_{i-1}(\au;X') \to \ldots.
	\]
	Let $\au$ as above a system of $t$ elements in $A$ and $b \in A$. Then the mapping cone construction 
	provides a long exact homology sequence 
	\[
	\ldots \to H_i(\au;X) \to H_i(\au;X) \to H_i(\au,b;X) \to H_{i-1}(\au;X) \to H_{i-1}(\au;X) \to \ldots, 
	\]
	where the homomorphism $H_i(\au;X) \to H_i(\au;X)$ is multiplication by $(-1)^i b$. \\
	(C) Let $(\cdot)^{\star} = \Hom_A(\cdot,A)$ the duality functor. Then we consider the co-Koszul complex $K^{\bullet}(\au;A)$ 
	defined by $\Hom_A(K_{\bullet}(\au;A),A) = (K_{\bullet}(\au;A))^{\star}$. Therefore the homomorphism 
	$$K^i(\au;A) \to K^{i+1}(\au;A)$$ is induced by $\Hom_A(d_{j_1 \ldots j_i},A)$ on the dual basis $(\bigwedge^i F)^{\star}$.
	It follows that $K_{\bullet}(\au;A)$ and $K^{\bullet}(\au;A)$ are isomorphic, that is, the Koszul complex 
	is self dual. Let $X$ denote an $A$-complex. Then 
	\[
	K^{\bullet}(\au;X) \cong \Hom_A(K_{\bullet}(\au;A),X) \text{ and } K^{\bullet}(\au;X) \cong K^{\bullet}(\au;A) \otimes_A X.
	\]
	We denote by $H^i(\au;X), i \in \mathbb{Z},$ the $i$-th cohomology of $K^{\bullet}(\au;X)$. We 
	have isomorphisms $H_i(\au;X) \cong H^{t-i}(\au;X)$ for all $i \in \mathbb{Z}$. Moreover $\au H_i(\au;X) = 0$ 
	for all $i \in \mathbb{Z}$. 
\end{remark}

For the proof of the last statement we recall the following well-known argument. 

\begin{lemma} \label{cone-1}
	Let $X$ denote a complex of $A$-modules. Let $a \in A$ denote an element.  Then $a H_i(a,X) = 0$ for 
	all $i \in \mathbb{Z}.$ 
\end{lemma}

\begin{proof} 
	By the construction of $K_{\bullet}(a;X)$ there is a short exact sequence 
	of complexes 
	\[
	0 \to X \to K_{\bullet}(a;X) \to X[-1] \to 0,
	\]
	where $X[-1]$ is the complex $X$ shifted the degrees by $-1$. The differential $\partial$ 
	on $K_i(a;X) = X_{i-1} \oplus X_i$ is given by $\partial_i(x,y) = ( d_{i-1}(x),d_i(y)+(-1)^{i-1}a x)$. 
	Suppose that $\partial_i(x,y) = 0$. That is  $d_{i-1}(x) = 0$ and $d_i(y) = (-1)^i ax$ and therefore 
	$a(x,y) = \partial_{i+1}((-1)^iy,0) \in \operatorname{Im} \partial_{i+1}$. This proves $a H_i(a;X) = 0$.
\end{proof}

In fact we shall use a slight modification of the above argument in further arguments. 

%
%

\begin{remark} \label{cech} ({\sl \v{C}ech complex.})
	(A) For a system of elements $\au = a_1,\ldots,a_t$ and an integer $n \geq 1$ we denote by $\au^n$ the system of elements
	$a_1^n,\ldots,a^n_t$. Then $\{K_{\bullet}(\au^n;A)\}_{n \geq 1}$ forms an inverse system of complexes and 
	$\{K^{\bullet}(\au^n;A)\}_{n \geq 1}$ forms a direct system of complexes. In both cases the maps 
	\[
	K_{\bullet}(\au^m;A) \to K_{\bullet}(\au^n;A) \text{ resp. } K^{\bullet}(\au^n;A) \to K^{\bullet}(\au^m;A)
	\]
	for $m \geq n$ are the naturally induced homomorphisms. Here we focus on the direct system $\{K^{\bullet}(\au^n;A)\}_{n \geq 1}$. 
	We start with a sequence consisting of a single element $a \in A$.  So there is the following 
	commutative diagram 
	\[
	\begin{array}{cccclcc}
		K^{\bullet}(a^n;A): \quad & 0 \to  & A & \stackrel{a^n}{\longrightarrow} & A & \to & 0 \\
		\downarrow          &   & \parallel &                         & \downarrow {\scriptstyle{a^{m-n}}} & & \\
		K^{\bullet}(a^m;A): \quad & 0 \to  & A & \stackrel{a^m}{\longrightarrow} & A & \to & 0
	\end{array}
	\] 
	for $m \geq n$. Its direct limit gives a complex $\check{C}_a : 0 \to A \to A_a \to 0$, where $A_a$ denotes 
	the localization of $A$ with respect to the multiplicatively closed set $\{a^n | n \geq 0\}$ with the 
	homomorphism $A \to A_a, r \mapsto r/1$. To this end recall 
	the fact that $\varinjlim \{A, a\} \cong A_a$, where the direct system is given by $A \stackrel{a}{\to} A$. \\
	(B) In general we define the \v{C}ech complex of a system of elements $\au = a_1,\ldots,a_t$ by 
	$\Cech = \varinjlim K^{\bullet}(\au^n;A)$. By elementary properties of direct limit and tensor products 
	it follows that $\Cech \cong \check{C}_{a_1} \otimes_A \cdots \otimes_A \check{C}_{a_t}$. As a consequence 
	there is the description 
	\[
	\Cech : 0 \to \Cech^0 \to \ldots \to \Cech^i \to \ldots \to \Cech^t \to 0 \; \text{ with } \;
	\Cech^i = \oplus_{1 \leq j_1 < \ldots < j_i \leq t} A_{a_{j_1}\cdots a_{j_i}},
	\]
	where the differential $d^i : \Cech^i \to \Cech^{i+1}$ is given at the component 
	$A_{a_{j_1}\cdots a_{j_i}} \to A_{a_{j_1}\cdots a_{j_{i+1}}}$ by $(-1)^{k+1}$ times the 
	natural map $A_{a_{j_1}\cdots a_{j_i}} \to (A_{a_{j_1}\cdots a_{j_i}})_{a_{j_k}}$ if $\{j_1, \ldots,j_i\} 
	= \{j_1,\ldots,\widehat{j_k}, \ldots,j_{i+1}\}$ and zero otherwise. For an $A$-complex $X$ we write 
	$\Cech(X) = \Cech \otimes_A X$. \\
	(C) The importance of the \v{C}ech complex is its relation to the local cohomology. Namely let 
	$\mathfrak{q} = (a_1,\ldots,a_t)A$ denote the ideal generated by the sequence $\au$. The local 
	cohomology $H^i_{\mathfrak{q}}(M)$ of an $A$-module $M$ is defined as the $i$-th right derived functor 
	$H^i_{\mathfrak{q}}(M)$ of the section functor $\Gamma_{\mathfrak{q}}(M) = \{ m \in M | \Supp_A Am \subseteq V(\mathfrak{q})\}$. 
	Then there are natural isomorphisms 
	\[
	H^i_{\mathfrak{q}}(M) \cong H^i(\Cech \otimes_A M) \cong \varinjlim H^i(\au^n;M)
	\]
	for an $A$-module $M$ and any $i \in \mathbb{Z}$. As a consequence it follows that $H^i(\Cech \otimes_A M)$ depends 
	only on the radical $\Rad \au R$.
\end{remark}	

For the details of the previous results we refer to \cite{BH} and \cite{pS0}. 

\section{The construction of complexes}
First we fix notations for this section. As above let $A$ denote a commutative Noetherian ring 
and $\mathfrak{q} \subseteq A$. Let $\au = a_1,\ldots,a_t$ denote a system of elements of $A$. 
Suppose that $a_i \in \mathfrak{q}^{c_i}$ for some integers $c_i \in \mathbb{N}$ for $i = 1,\ldots,t$. 
Let $M$ denote a finitely generated $A$-module. 

\begin{notation} \label{kos-2}
	Let $n$ denote an integer. We define a complex $K_{\bullet}(\au,\mathfrak{q},M;n)$ in the following way:
	\begin{itemize}
		\item[(a)] For $0 \leq i \leq t$ put $K_i(\au,\mathfrak{q},M;n) = \oplus_{1\leq j_1 < \ldots < j_i \leq t} 
		\mathfrak{q}^{n-c_{j_1}- \ldots- c_{j_i}}M$ and $K_i(\au,\mathfrak{q},M;n) = 0$ for $i >t$ or $i <0$.
		\item[(b)] The boundary map $K_i(\au,\mathfrak{q},M;n) \to K_{i-1}(\au,\mathfrak{q},M;n)$ is defined by maps on each of 
		the direct summands $\mathfrak{q}^{n-c_{j_1}- \ldots- c_{j_i}}M$. On $\mathfrak{q}^{n-c_{j_1}- \ldots- c_{j_i}}M$ 
		it is the map given by $d_{j_1 \ldots j_i} \otimes 1_M$ restricted to $\mathfrak{q}^{n-c_{j_1}- \ldots- c_{j_i}}M$, where $d_{j_1 \ldots j_i}$ denotes the homomorphism as defined in \ref{kos-1}.
	\end{itemize}
	It is clear that the image of the map is contained in $\oplus_{1\leq j_1 < \ldots < j_{i-1} \leq t} 
	\mathfrak{q}^{n-c_{j_1}- \ldots- c_{j_{i-1}}}M$. Clearly it is a boundary homomorphism. 
	By the construction it follows that $K_{\bullet}(\au,\mathfrak{q},M;n)$ is a sub complex of the 
	Koszul complex $K_{\bullet}(\au;M)$ for each $n \in \mathbb{N}$.
\end{notation}

Another way for the construction is the following.

\begin{remark} \label{rem-kos}
	Let $R_A(\mathfrak{q})$ and $R_M(\mathfrak{q})$ denote the Rees ring and the Rees module. 
	For $a_i, i = 1,\ldots,t,$ we consider $a_i T^{c_i} \in [R_A(\mathfrak{q})]_{c_i}$. Then we have 
	the system $\underline{aT^c} = a_1T^{c_1}, \ldots,a_tT^{c_t}$ of elements of $R_A(\mathfrak{q})$. Note 
	that $\deg a_iT^{c_i} = c_i, i = 1,\ldots,t$. Then we may consider the Koszul complex 
	$K_{\bullet}(\underline{aT^c};R_M(\mathfrak{q}))$. This is a complex of graded $R_A(\mathfrak{q})$-modules. 
	It is easily seen that the degree $n$-component $[K_{\bullet}(\underline{aT^c};R_M(\mathfrak{q}))]_n$ 
	of $K_{\bullet}(\underline{aT^c};R_M(\mathfrak{q}))$ is the complex $K_{\bullet}(\au,\mathfrak{q},M;n)$ as introduced 
	in \ref{kos-2}. We write $H_i(\au,\mathfrak{q},M;n)$ for the $i$-th 
	homology of $K_{\bullet}(\au,\mathfrak{q},M;n)$ for $i \in \mathbb{Z}$.
\end{remark}

We come now to the definition of one of the main subjects of the paper.

\begin{definition} \label{def-1}
	With the previous notation we define $\mathcal{L}_{\bullet}(\au,\mathfrak{q},M;n)$ the quotient of the 
	embedding $K_{\bullet}(\au,\mathfrak{q},M;n) \to K_{\bullet}(\au;M)$. That is there is a short exact sequence 
	of complexes
	\[
	0 \to K_{\bullet}(\au,\mathfrak{q},M;n) \to K_{\bullet}(\au;M) \to \mathcal{L}_{\bullet}(\au,\mathfrak{q},M;n) 
	\to 0.
	\]
	Note that $\mathcal{L}_i(\au,\mathfrak{q},M;n) \cong \oplus_{1 \leq j_1 < \ldots < j_i \leq t}M/\mathfrak{q}^{n-c_{j_1}- \ldots- c_{j_i}}M$.
	The boundary maps are those induced by the Koszul complex. We write $L_i(\au,\mathfrak{q},M;n)$ 
	for the $i$-th homology of $\mathcal{L}_{\bullet}(\au,\mathfrak{q},M;n)$ and any $i \in \mathbb{Z}$.
\end{definition}

For a construction by mapping cones we need the following technical result. For a morphism $f: X \to Y$ we write 
$C(f)$ for the mapping cone of $f$.

\begin{lemma} \label{cone-2}
	With the previous notation let $b \in \mathfrak{q}^d$ an element. The multiplication map by $b$ induces 
	the following morphisms 
	\[
	m_b(K) : K_{\bullet}(\au,\mathfrak{q},M;n-d) \to K_{\bullet}(\au,\mathfrak{q},M;n) \text{ and } 
	m_b(\mathcal{L}) : \mathcal{L}_{\bullet}(\au,\mathfrak{q},M;n-d) \to \mathcal{L}_{\bullet}(\au,\mathfrak{q},M;n)
	\] 
	of complexes. They induce isomorphism of complexes 
	\[
	C(m_b(K)) \cong  K_{\bullet}(\au,b,\mathfrak{q},M;n) \text{ and } C(m_b(\mathcal{L})) \cong \mathcal{L}_{\bullet}(\au,b,\mathfrak{q},M;n).
	\]
\end{lemma}

\begin{proof}
	The proof follows easily by the structure of the complexes and the mapping cone construction. 
\end{proof}

We begin with a few properties of the previous complexes. 

\begin{theorem} \label{kos-3}
	Let $\au = a_1,\ldots,a_t$ denote a system of elements of $A$, $\mathfrak{q} \subset A$ an ideal and $M$ a 
	finitely generated $A$-module. Let $n \in \mathbb{N}$ denote an integer. 
	\begin{itemize}
		\item[(a)] $H_i(\au,\mathfrak{q},M;n) \cong H_i(\au_{\sigma},\mathfrak{q},M;n)$ and 
		$L_i(\au,\mathfrak{q},M;n) \cong L_i(\au_{\sigma},\mathfrak{q},M;n)$ for all $i \in \mathbb{Z}$ and any $\sigma$,
		a permutation on $t$ letters.
		\item[(b)] $\au H_i(\au,\mathfrak{q},M;n) = 0$ and $\au L_i(\au,\mathfrak{q},M;n) = 0$ for all $i \in \mathbb{Z}$.
		\item[(c)] $H_i(\au,\mathfrak{q},M;n)$ and $L_i(\au,\mathfrak{q},M;n)$ are finitely generated $A/\au A$-modules 
		for all $i \in \mathbb{Z}$.
	\end{itemize}
\end{theorem}

\begin{proof}
	The statement in (a) follows by virtue of the short exact sequence of complexes in \ref{def-1} and the long exact 
	homology sequence. Note that the homology of Koszul complexes is isomorphic under permutations. 
	
	The claim in (c) is a consequence of (b) since the homology modules $H_i(\au,\mathfrak{q},M;n)$ and $L_i(\au,\mathfrak{q},M;n)$ 
	are finitely generated $A$-modules. 
	
	For the proof of (b) we follow the mapping cone construction of \ref{cone-2} with the arguments of \ref{cone-1}. To this end 
	let $K_n = K_{\bullet}(\au,\mathfrak{q},M;n)$ and $C = C(m_b(K)_n) = K_{\bullet}(\au,b,\mathfrak{q},M;n)$. Then there is a short exact sequence 
	of complexes 
	\[
	0 \to K_n \to C \to K_{n-d}[-1] \to 0. 
	\]
	The differential $\partial_i$ on $(x,y) \in C_i = (K_{n-d})_{i-1} \oplus (K_n)_i$ is given by 
	\[
	\partial_i(x,y) = (d_{i-1}(x),d_i(y) +(-1)^{i-1}b x).
	\]
	Suppose that $\partial_i(x,y) = 0$, i.e., $d_{i-1}(x) =0$ and $d_i(y) = (-1)^i bx$. Then 
	$$(y,0) \in (K_{n-d})_i\oplus (K_n)_{i+1} =C_{i+1} \text{ because } (K_n)_i \subseteq (K_{n-d})_i$$ 
	and therefore $\partial_{i+1}((-1)^iy,0) = b(x,y)$. That is $b H_i(C) = 0$ for all $i \in \mathbb{Z}$. 
	
	In order to show the claim in (b) we use the previous argument. So let us consider 
	$K_{\bullet}(\au,\mathfrak{q},M;n) = C(m_{a_t}(K_{\bullet}(\au',\mathfrak{q},M;n)))$, where 
	$\au' = a_1,\ldots,a_{t-1}$. The previous argument shows $a_t H_i(\au,\mathfrak{q},M;n) = 0$ for all 
	$i \in \mathbb{Z}.$ By view of (a) this finishes the proof in the case of $H_i(\au,\mathfrak{q},M;n)$. 
	
	For the proof of $\au L_i(\au,\mathfrak{q},M;n) = 0$ we follow the same arguments. Instead 
	of the injection 
	$(K_n)_i \subseteq (K_{n-d})_i$ we use the surjection $(\mathcal{L}_n)_i \twoheadrightarrow (\mathcal{L}_{n-d})_i$, where 
	$\mathcal{L}_n = \mathcal{L}_{\bullet}(\au,\mathfrak{q},M;n)$. We skip the details here.
\end{proof}

\section{The construction of co-complexes}
With the notations of the previous section we shall define the co-complex version of the complexes above. 
This is based on the Koszul co-complex. 

\begin{notation} \label{kos-4}
	For an integer $n \in \mathbb{N}$ we define a complex $K^{\bullet}(\au,\mathfrak{q},M;n)$ 
	similar to the construction in \ref{kos-2} by the use of the Koszul co-complex. By view of 
	Remark \ref{rem-kos} we  define it also as the $n$-th graded component of $K^{\bullet}(\underline{aT^c};R_M(\mathfrak{q}))$.
	We may identify $K^i(\underline{aT^c};R_M(\mathfrak{q}))_n$ with $\oplus_{1 \leq j_1 < \ldots < j_i \leq t} 
	\mathfrak{q}^{n +c_{j_1} + \ldots + c_{j_i}}M$. 
	It follows that $K^{\bullet}(\underline{aT^c};R_M(\mathfrak{q}))_n$ is a subcomplex of $K^{\bullet}(\au;M)$. 
	We define $\mathcal{L}^{\bullet}(\au,\mathfrak{q},M;n)$ as the quotient of this embedding. That is, there is a short 
	exact sequence of complexes 
	\[
	0 \to K^{\bullet}(\au,\mathfrak{q},M;n) \to K^{\bullet}(\au;M) \to \mathcal{L}^{\bullet}(\au,\mathfrak{q},M;n)
	\to 0.
	\] 
	We may identify $\mathcal{L}^i(\au,\mathfrak{q},M;n)$ with $\oplus_{1 \leq j_1 < \ldots < j_i \leq t} 
	M/\mathfrak{q}^{n + c_{j_1} + \ldots + c_{j_i}} M$. 	
	We denote by $H^i(\au,\mathfrak{q},M;n)$ 
	resp. $L^i(\au,\mathfrak{q},M;n)$ the $i$-th cohomology of $K^{\bullet}(\au,\mathfrak{q},M;n)$ 
	resp. $\mathcal{L}^{\bullet}(\au,\mathfrak{q},M;n)$ for any $i \in \mathbb{Z}$.
\end{notation}

For an iteration we need the following technical result. For a morphism $f: X \to Y$ of co-complexes we write 
$D(f)$ for the co-mapping cone of $f$.

\begin{lemma} \label{cone-3}
	With the previous notation let $b \in \mathfrak{q}^d$ an element. The multiplication map by $b$ induces 
	the following morphisms 
	\[
	m_b(K) : K^{\bullet}(\au,\mathfrak{q},M;n) \to K^{\bullet}(\au,\mathfrak{q},M;n+d) \text{ and } 
	m_b(\mathcal{L}) : \mathcal{L}^{\bullet}(\au,\mathfrak{q},M;n) \to \mathcal{L}^{\bullet}(\au,\mathfrak{q},M;n+d)
	\] 
	of complexes. They induce isomorphism of complexes 
	\[
	D(m_b(K)) \cong  K^{\bullet}(\au,b,\mathfrak{q},M;n) \text{ and } D(m_b(\mathcal{L})) \cong \mathcal{L}^{\bullet}(\au,b,\mathfrak{q},M;n).
	\]
\end{lemma}

\begin{proof}
	This is easy by reading of the definitions. 
\end{proof}

\begin{theorem} \label{kos-5}
	Let $\au = a_1,\ldots,a_t$ denote a system of elements of $A$, $\mathfrak{q} \subset A$ an ideal and $M$ a 
	finitely generated $A$-module. Let $n \in \mathbb{N}$ denote an integer. 
	\begin{itemize}
		\item[(a)] $H^i(\au,\mathfrak{q},M;n) \cong H^i(\au_{\sigma},\mathfrak{q},M;n)$ and 
		$L^i(\au,\mathfrak{q},M;n) \cong L^i(\au_{\sigma},\mathfrak{q},M;n)$ for all $i \in \mathbb{Z}$ and any $\sigma$,
		a permutation on $t$ letters.
		\item[(b)] $\au H^i(\au,\mathfrak{q},M;n) = 0$ and $\au L^i(\au,\mathfrak{q},M;n) = 0$ for all $i \in \mathbb{Z}$.
		\item[(c)] $H^i(\au,\mathfrak{q},M;n)$ and $L^i(\au,\mathfrak{q},M;n)$ are finitely generated $A/\au A$-modules 
		for all $i \in \mathbb{Z}$.
	\end{itemize}
\end{theorem}

\begin{proof}
	The arguments in the proof are a repetition of those of the proof of Theorem \ref{kos-3} with cohomology 
	instead of homology. We skip the details here.
\end{proof}

\section{Euler characteristics}
Let $A$ denote a commutative ring. Let $X$ denote a complex of $A$-modules. 

\begin{definition} \label{def-2}
	Let $X : 0 \to X_n \to \ldots \to X_1 \to X_0 \to 0$ denote a bounded complex of $A$-modules. 
	Suppose that $H_i(X), i = 0, 1, \ldots, n,$ is an $A$-module of finite length. Then 
	\[
	\chi_A(X) = \sum_{i=0}^n (-1)^i \ell_A(H_i(X))
	\]
	is called the Euler characteristic of $X$.
\end{definition}

We collect a few well known facts about Euler characteristics. 

\begin{lemma} \label{char}
	Let $A$ denote a Noetherian commutative ring. 
	\begin{itemize}
		\item[(a)] Let $0 \to X' \to X \to X'' \to 0$ denote a short exact sequence of complexes such that all the homology 
		modules are of finite length. Then 
		$
		\chi_A(X) = \chi_A(X') + \chi_A(X'').
		$
		\item[(b)] Suppose $X: 0 \to X_n \to \ldots \to X_1 \to X_0 \to 0$ is a bounded complex such 
		that $X_i, i = 0,\ldots,n,$ is of finite length. Then 
		$
		\chi_A(X) = \sum_{i=0}^n (-1)^i \ell_A(X_i)
		$
	\end{itemize}
\end{lemma}

\begin{proof}
	The statement in (a) follows by the long exact cohomology sequence derived by $0 \to X' \to X \to X'' \to 0$. The 
	second statement might be proved by induction on $n$, the length of the complex $X$.
\end{proof}

As an application we get the following result about multiplicities, originally shown by 
\cite{jpS} and \cite{AB}.

\begin{proposition} \label{mult-1}
	Let $(A,\mathfrak{m})$ be a local ring and $a_1,\ldots,a_d \in \mathfrak{m}$ a system of parameters 
	for $M$, a finitely generated $A$-module. Then
	\[
	\chi_A(\au;M) = e_0(\au;M),
	\]
	where $\chi_A(\au;M) = \chi_A(K_{\bullet}(\au;M))$ and $e_0(\au;M)$ denotes the Hilbert-Samuel 
	multiplicity.  
\end{proposition}

\begin{proof} Let $\au = a_1,\ldots,a_d$ the system of parameters and $\au A = \mathfrak{q}$. 
	We choose $c_i = 1, i = 1,\ldots,d$. Then the short 
	exact sequence of \ref{def-1} has the following form 
		\[
		0 \to K_{\bullet}(\au,\au,M;n) \to K_{\bullet}(\au;M) \to \mathcal{L}_{\bullet}(\au,\au,M;n) \to 0.
		\]
	All of the three complexes have homology modules of finite length and therefore 
	$\chi_A(\au;M) = \chi_A(K_{\bullet}(\au,\au,M;n))+ \chi_A(\mathcal{L}_{\bullet}(\au,\au,M;n))$ for all $n \in \mathbb{N}$.
	
	First we show that $\chi_A(K_{\bullet}(\au,\au,M;n)) = 0$ for all $n \gg 0$. To this end recall that 
	$H_i(\au,\au,M;n)) = [H_i(\underline{aT};R_M(\au))]_n$. We know that $\underline{aT} (H_i(\underline{aT};R_M(\au)) = 0$ for all 
	$i = 0, \ldots,d$. Therefore $H_i(\underline{aT};R_M(\au))$ is a finitely generated module over 
	$R_A(\au)/\underline{aT} R_A(\au) = A/\au A$. This implies that $[H_i(\underline{aT};R_M(\au))]_n = H_i(\au,\au,M;n) = 0$ 
	for all $n \gg 0$. That is $\chi_A(K_{\bullet}(\au,\au,M;n)) = 0$ for all $n \gg 0$. 
	
	By view of Lemma \ref{char} (b) we get $\chi_A(\mathcal{L}_{\bullet}(\au,\au,M;n)) = \sum_{i=0}^d (-1)^i\binom{d}{i}
	\ell_A(M/\au^{n-i}M)$. For $n \gg 0$ the length $\ell_A(M/\au^nM)$ is given by the Hilbert polynomial 
	$e_0(\au;M) \binom{d+n}{d} + \ldots + e_d(\au;M)$. Therefore, it follows that $\chi_A(\mathcal{L}_{\bullet}(\au,\au,M;n)) =
	e_0(\au;M)$. This completes the argument.
\end{proof}

The more general situation of a system of parameters $\au = a_1,\ldots, a_d$ of a finitely generated $A$-module 
$M$ of a local ring $(A,\mathfrak{m})$ and an ideal $\mathfrak{q} \supset \au$ with $a_i \in \mathfrak{q}^{c_i},
i = 1,\ldots,d$ is investigated in the following.

\begin{proposition} \label{mult-2}
	With the previous notation we have the equality 
   \[
   e_0(\au;M) = c_1 \cdots c_d e_0(\mathfrak{q};M) + \chi_A(K_{\bullet}(\au,\mathfrak{q},M;n))
   \]
   for all $n \gg 0$. In particular, for all $n \gg 0$ the Euler characteristic $\chi_A(K_{\bullet}(\au,\mathfrak{q},M;n))$ 
   is a constant.
\end{proposition}

\begin{proof}
	The proof follows by the inspection of the short exact sequence of complexes 
	\[
	0 \to K_{\bullet}(\au,\mathfrak{q},M;n) \to K_{\bullet}(\au;M) \to \mathcal{L}_{\bullet}(\au,\mathfrak{q},M;n) \to 0
	\]
	of \ref{def-1}. By view of Proposition \ref{mult-1} we have $e_0(\au;M)$ for the Euler characteristic 
	of the complex in the middle. For the Euler characteristic on the right we get (see \ref{char} (b))
	\[
	\chi_A(\mathcal{L}_{\bullet}(\au,\mathfrak{q},M;n)) = \sum_{i=0}^{d} (-1)^i \sum_{1 \leq j_1 < \ldots < j_i \leq d}
	\ell_A(M/\mathfrak{q}^{n-c_{j_1}-\ldots- c_{j_i}}M),
	\]
	which gives the first summand in the above formula (see also \cite{BS} for the details in the case of $M = A$).
	This finally proves the claim.
\end{proof}

For several reasons it would be interesting to have an answer to the following problem.

\begin{problem} \label{prob}
	With the notation of Proposition \ref{mult-2} it would be of some interest to give 
	an interpretation of $\chi_A(\au,\mathfrak{q},M) := \chi_A(K_{\bullet}(\au,\mathfrak{q},M;n))$ for 
	large $n \gg 0$ independently of $n$. By a slight 
	modification of an argument given in \cite{BS} it follows that $\chi_A(\au,\mathfrak{q},M) \geq 0$. 
\end{problem}

\section{The modified \v{C}ech complexes}
Before we shall be concerned with the construction of our complexes we need a technical lemma. 
To this end let $A$ denote a commutative Noetherian ring and let $\mathfrak{q} \subset A$ be an ideal. 
Let $M$ denote a finitely generated $A$-module.
Let $a \in \mathfrak{q}^c$ denote an element. Then for each integer $n \in \mathbb{N}$ the multiplication by $a$ induces a map 
\[
\mathfrak{q}^n M \to \mathfrak{q}^{n+c} M, \; m \mapsto am.
\]
By iterating this map there is a direct system $\{\mathfrak{q}^{n+kc} M, a\}$, where 
$\mathfrak{q}^{n+kc} M \to \mathfrak{q}^{n+(k+1)c} M$ is the multiplication by $a \in \mathfrak{q}^c$.

\begin{lemma} \label{dirlim}
	Let $n \in \mathbb{N}$ be an integer. With the previous notation there is an isomorphism 
	\[
	\varinjlim_k \{\mathfrak{q}^{n+kc} M, a\} \cong \mathfrak{q}^n M[\mathfrak{q}^c/a],	
	\]
	where $M[\mathfrak{q}^c/a] \subseteq M_a$ consists of all elements of the form 
	$m_0/1+q_1 m_1/a + \ldots + q_t m_t/a^t$ for some $t \geq 0$ and elements $m_0,\ldots,m_t \in M$ and 
	$q_i \in \mathfrak{q}^{ci}, i = 1,\ldots,t$.
\end{lemma}

\begin{proof} For the proof we use the direct system $\{M, a\}$ with its direct limit 
	$\varinjlim \{M,a\} \cong M_a$. The injection $\mathfrak{q}^{n+kc}M \to M$ provides an injection of direct systems 
	$\{\mathfrak{q}^{n+kc} M, a\} \to \{M, a\}$. By the definition of the direct limit there is a 
	commutative diagram with exact rows
	\[
	\begin{array}{cccccccc}
	0 & \to & \oplus \,\mathfrak{q}^{n+kc}M & \to &  \oplus \,\mathfrak{q}^{n+kc}M & \to & \varinjlim \{\mathfrak{q}^{n+kc} M, a\} & \to 0\\
	  &     &   \downarrow                &     &   \downarrow                 &     & \downarrow & \\
	0 & \to & \oplus \, M                    & \to & \oplus \, M & \to & M_a & \to 0,
	\end{array}
	\]
	where the vertical maps are injections. Let $x \in \varinjlim \{\mathfrak{q}^{n+kc} M, a\}$. Then it may be 
	written as $x = q_0 m_0/1 + q_1m_1/a + \ldots + q_tm_t/a^t$ with $m_i \in M$ and $q_i \in \mathfrak{q}^{n+ic}$ 
	for $i = 0,\ldots,t$. Then 
	\[
	x= 1/a^t(q_0 a^t m_0 + q_1 a^{t-1}m_1+ \ldots + q_tm_t) \in (\mathfrak{q}^n/a^t) (\mathfrak{q}^c,a)^t M .
	\] 
	Because of $a \in \mathfrak{q}^c$ it follows $x \in (\mathfrak{q}^n) ((\mathfrak{q}^{c}/a)^t M) \subseteq \mathfrak{q}^n
	M[\mathfrak{q}^c/a]$. The reverse containment follows the same line of reasoning. 
\end{proof}

With these preparations we are ready to define certain modifications of the \v{C}ech complexes.
As above let $\au = a_1,\ldots,a_t$ denote a system of elements with the previous notations. 

\begin{notation} \label{cech-1}
	(A) Let $n \in \mathbb{N}$ denote an integer. We define a complex $\check{C}^{\bullet}(\au,\mathfrak{q},M;n)$ in the 
	following way:
	\begin{itemize}
		\item[(a)] For $0 \leq i \leq t$ put 
		$$
		\check{C}^i(\au,\mathfrak{q},M;n) = 
		\oplus_{1 \leq j_1 < \ldots < j_i \leq t} \mathfrak{q}^n M[\mathfrak{q}^{c_{j_1}+ \ldots + c_{j_i}}/(a_{j_1}\cdots a_{j_i})] \subseteq \oplus_{1 \leq j_1 < \ldots < j_i \leq t} M_{a_{j_1}\cdots a_{j_i}}
		$$
		and $\check{C}^i(\au,\mathfrak{q},M;n) = 0$ for $i > t$ or $i < 0$.
		\item[(b)] The boundary map $\check{C}^i(\au,\mathfrak{q},M;n) \to \check{C}^{i+1}(\au,\mathfrak{q},M;n)$ 
		is the restriction of the boundary map $\check{C}^i_{\au} \otimes_A M \to \check{C}^{i+1}_{\au} \otimes_A M$.
	\end{itemize}
	Note that the restriction is really a boundary map on $\check{C}^{\bullet}(\au,\mathfrak{q},M;n)$.
	We write $\check{H}^i(\au,\mathfrak{q},M;n)$ for the $i$-th cohomology of $
	\check{C}^{\bullet}(\au,\mathfrak{q},M;n)$ for all integers $i \in \mathbb{Z}$ and $n \in \mathbb{N}$.\\
	(B) By the construction it is clear that 
	$$
	\check{C}^{\bullet}(\au,\mathfrak{q},M;n) \to \check{C}_{\au} \otimes_A M
	$$ 
	is a subcomplex. We write $\check{\mathcal{L}}^{\bullet}(\au,\mathfrak{q},M;n)$ for the quotient. Whence 
	there is a short exact sequence of complexes
	\[
	0 \to \check{C}^{\bullet}(\au,\mathfrak{q},M;n) \to \check{C}_{\au} \otimes_A M \to 
	\check{\mathcal{L}}^{\bullet}(\au,\mathfrak{q},M;n) \to 0.
	\]
	We write $\check{L}^i(\au,\mathfrak{q},M;n)$ for the $i$-th cohomology of 
	$\check{\mathcal{L}}^{\bullet}(\au,\mathfrak{q},M;n)$ for all $i \in \mathbb{Z}$ and $n \in \mathbb{N}$.
\end{notation}

Next we relate the construction more closely to the \v{C}ech complexes. To this end we put 
$\au^k = a_1^k,\ldots,a_t^k$ in $A$ and $\underline{a^kT^{ck}} = a_1^kT^{c_1k},\ldots,a_t^kT^{c_tk}$ 
in $R_A(\mathfrak{q})$. 

\begin{proposition} \label{cech-2}
	With the previous notation there are isomorphisms of complexes
	\[
	\check{C}^{\bullet}(\au,\mathfrak{q},M;n) \cong \varinjlim_{k} 
	K^{\bullet}(\au^k,\mathfrak{q},M;n) \,
	\text{ and } \, 
	\check{\mathcal{L}}^{\bullet}(\au,\mathfrak{q},M;n) \cong \varinjlim_{k} \mathcal{L}^{\bullet}(\au^k,\mathfrak{q},M;n)
	\]
	for all  $n \in \mathbb{N}$.
\end{proposition}

\begin{proof}
	By the definitions (see \ref{kos-2}) there are isomorphisms 
	$$
	K^{\bullet}(\underline{a^kT^{ck}};R_M(\mathfrak{q}))_n \cong 
	K^{\bullet}(\au^k,\mathfrak{q},M;n)
	$$
	for all $k \geq 1$. These isomorphisms are compatible with the homomorphisms of the 
	corresponding direct systems. By passing to the direct limit there are isomorphisms 
	\[
	\check{C}_{\underline{aT^c}}(R_M(\mathfrak{q}))_n \cong \varinjlim_{k} 
	K^{\bullet}(\au^k,\mathfrak{q},M;n).
	\]
	By an inspection of $K^{\bullet}(\au^k,\mathfrak{q},M;n)$ as a subcomplex of $K^{\bullet}(\au^k;M)$ 
	(see \ref{kos-2}) it follows by virtue of \ref{dirlim} that 
	$\varinjlim_{k} K^{\bullet}(\au^k,\mathfrak{q},M;n) \cong \check{C}^{\bullet}(\au,\mathfrak{q},M;n)
	$. To this end recall that 
	\[
	\mathfrak{q}^{n+k(c_{j_1}+ \ldots + c_{j_i})}M \to \mathfrak{q}^{n+(k+1)(c_{j_1}+ \ldots + c_{j_i})}M
	\]
	is the multiplication by $a_{j_1} \cdots a_{j_i}$ on the $(j_1,\ldots,j_i)$-component on 
	$K^i(\au^k,\mathfrak{q},M;n)$ for all $1 \leq j_1 < \ldots < j_i \leq t$.
	
	By virtue of the notation in \ref{kos-2} there is a short exact sequence of complexes
	\[
	0 \to K^{\bullet}(\underline{a^kT^{ck}};R_M(\mathfrak{q}))_n \to K^{\bullet}(\au^k;M) \to \mathcal{L}^{\bullet}(\au^k,\mathfrak{q},M;n)
	\to 0
	\]
	for all $n \geq 1$. By passing to the direct limit the first isomorphism 
	provides the second one by the definitions. 
\end{proof}

As an application we show that the cohomology $\check{L}^i(\au,\mathfrak{q},M;n)$ depends only upon the 
radical $\Rad \au$ of the ideal $\au A$.

\begin{lemma} \label{cech-3}
	Let $\au = a_1,\ldots,a_k$ and $\underline{b}=  b_1,\ldots,b_l$ be two sequences of elements of $A$. 
	Suppose that $a_i \in \mathfrak{q}^{c_i}, i = 1,\ldots,k$ and $b_j \in \mathfrak{q}^{d_j}, j = 1,\ldots,l$. 
	If $\Rad \au A = \Rad \underline{b} R$, then there are isomorphisms
	\[
	\check{C}^{\bullet}(\au,\mathfrak{q},M;n) \cong \check{C}^{\bullet}(\underline{b},\mathfrak{q},M;n) 
	\text{ and } 
	\check{\mathcal{L}}^{\bullet}(\au,\mathfrak{q},M;n) \cong \check{\mathcal{L}}^{\bullet}(\underline{b},\mathfrak{q},M;n)
	\]
	for all $ n \in \mathbb{N}$.
\end{lemma}

\begin{proof}
	First put $\underline{bT^d} = b_1T^{d_1},\ldots,b_lT^{d_l}$. Then we claim that 
	$\Rad \underline{aT^c}R_A(\mathfrak{q}) = \Rad \underline{bT^d} R_A(\mathfrak{q})$. Let $b \in \Rad \au A$, 
	that is $b \in \mathfrak{q}^{e}$ and $b^m = \sum_{i=1}^k a_i r_i$. That is $(bT^e)^m = 
	\sum_{i=1}^k a_iT^{c_i}r_i T^{me-c_i} \in \underline{aT^c} R_A(\mathfrak{q}).$ The reverse inclusion 
	can be proved similarly. 
	
	By passing to the direct limit of the short exact sequence at the end of the proof of 
	\ref{cech-2} provides a short exact sequence of complexes
	\[
	0 \to \check{C}_{\underline{aT^c}}(R_M(\mathfrak{q}))_n \to \check{C}_{\au}(M) \to 
	\check{\mathcal{L}}^{\bullet}(\au,\mathfrak{q},M;n) \to 0.
	\] 
	The \v{C}ech complexes are isomorphic for ideals equal up to the radical. By construction we have 
	$\check{C}_{\underline{aT^c}}(R_M(\mathfrak{q}))_n \cong \check{C}^{\bullet}(\au,\mathfrak{q},M;n)$, 
	which proves the first isomorphism. Since $ \check{C}_{\au}\otimes_AM \cong \check{C}_{\underline{b}} 
	\otimes_A M$ the previous sequence proves the second isomorphism of the statement. 
\end{proof}

\begin{remark} \label{cech-4}
	By our construction it follows that $\check{C}^{\bullet}(\au,\mathfrak{q},M;n) 
	\cong [\check{C}_{\underline{aT^c}}(R_M(\mathfrak{q}))]_n$ for all $n \in \mathbb{Z}$. That is 
	the cohomology of $\check{C}^{\bullet}(\au,\mathfrak{q},M;n)$ is the $n$-th graded component 
	of the cohomology of $\check{C}_{\underline{aT^c}}(R_M(\mathfrak{q}))$. In the particular case of 
	$\mathfrak{q} = \au A$ and $c_1= \ldots= c_t =1$ there is an interpretation in sheaf cohomology. 
	To this end let $X = \Proj R_A(\au)$ and $\mathcal{F} = (R_M(\au))^{\sim}$ the associated 
	$\mathcal{O}_X$-module. Then for all $n \in \mathbb{Z}$ there is an exact sequence 
	\[
	0 \to H^0_{\au}(R_M(\au))_n \to R_M(\au)_n \to H^0(X,\mathcal{F}(n)) \to H^1_{\au}(R_M(\au))_n \to 0
	\]
	and isomorphisms $H^i(X,\mathcal{F}(n)) \cong H^{i+1}_{\au}(R_M(\au))_n$ for all $i \geq 1$. 
	Then it follows that $R_M(\au)_n \cong H^0(X,\mathcal{F}(n))$ and $H^i(X,\mathcal{F}(n)) = 0$ 
	for all $n \gg 0$ and $i \geq 1$. For these and related results about sheaf cohomology we refer to 
	Hartshorne's book \cite{rH}.  
	
	By view of the short exact sequence of complexes in \ref{cech-1} (B) it turns out that 
	$H^i_{\au}(M) \cong \check{L}^i(\au,\au,M;n)$ for all $ n \gg 0$ and $i \geq 0$.
\end{remark}

\section{Sequences}
Here let $(A,\mathfrak{m},\Bbbk)$ denote a local ring.
Let $\mathfrak{a}$ denote an ideal of a Noetherian ring $A$. For a finitely generated $A$-module $M$ and 
$N \subset M$ let $N :_M \langle \mathfrak{a}\rangle$ denote the stable value of $N :_M \mathfrak{a}^m$ 
for $m \gg 0$. 

\begin{lemma} \label{seq-1}
	With the previous notation let $\mathcal{N} = 0:_{R_M(\mathfrak{q})} \langle \underline{aT^c}\rangle$.  Then 
	\begin{itemize}
		\item[(a)] $\mathcal{N} = \oplus_{n \geq 0} (0:_M\langle \au \rangle) \cap \mathfrak{q}^nM$, 
		\item[(b)] $H^0_{\underline{aT^c}}(R_M(\mathfrak{q})/\mathcal{N}) = 0$,  
		\item[(c)] $H^0_{\underline{aT^c}}(R_M(\mathfrak{q})) \cong \mathcal{N}$ and 
		$H^i_{\underline{aT^c}}(R_M(\mathfrak{q})) \cong H^i_{\underline{aT^c}}(R_{M/0:_M \langle \au \rangle}(\mathfrak{q}))$ for all $i > 0$.
	\end{itemize}
\end{lemma}

\begin{proof}
	The proof of (a) is clear since $\mathcal{N} = \oplus_{n \geq 0} 0 :_{\mathfrak{q}^n M} \langle \au \rangle$. 
	Moreover we have a short exact sequence 
	\[
	0 \to \mathcal{N} \to R_M(\mathfrak{q}) \to R_{M/0:_M\langle \au \rangle}(\mathfrak{q}) \to 0.
	\]
	To this end note that $[ R_{M/0:_M\langle \au \rangle}(\mathfrak{q})]_n = 
	\mathfrak{q}^n(M/0:_M\langle \au \rangle) \cong \mathfrak{q}^nM/(0:_M\langle \au \rangle) \cap \mathfrak{q}^nM$ 
	for all $n \geq 0$. Then we apply the local cohomology functor $H^i_{\underline{aT^c}}(\cdot)$. The statements 
	in (b) and (c) follow now since $\mathcal{N}$ is of $\underline{aT^c}$-torsion.
\end{proof}

In the following we will investigate the behaviour modulo a certain element. 

\begin{lemma} \label{seq-2}
	With the previous notation let $a \in A$ denote an $M$-regular element with $a \in \mathfrak{q}^c \setminus 
	\mathfrak{q}^{c+1}$ such that $a^{\star}= a + \mathfrak{q}^{c+1} \in G_A(\mathfrak{q})$ has the property 
	that $a^{\star} \notin \mathfrak{P}$ for all $\mathfrak{P} \in \Ass G_M(\mathfrak{q})$ with 
	$\mathfrak{P} \not\supset G_A(\mathfrak{q})_+$. Put $\mathcal{M} = \oplus_{n \geq 0} 
	aM \cap \mathfrak{q}^n M /a \mathfrak{q}^{n-c} M$. Then there is a short exact sequence
	\[
	0 \to\mathcal{M} \to H^0_{\underline{aT^c}}(R_M(\mathfrak{q})/(aT^c)R_M(\mathfrak{q})) \to 
	H^0_{\underline{aT^c}}(R_{M/aM}(\mathfrak{q})) \to 0
	\]
	and isomorphisms $ H^i_{\underline{aT^c}}(R_M(\mathfrak{q})/(aT^c)R_M(\mathfrak{q})) \cong 
		H^i_{\underline{aT^c}}(R_{M/aM}(\mathfrak{q}))$ for all $i > 0$.
\end{lemma}

\begin{proof}
	By the choice of $a^{\star}$ it follows that 
	$$
	0 :_{G_M(\mathfrak{q})} a^{\star} = \oplus_{n \geq 0} 
	(\mathfrak{q}^{n+c+1}M :_M a)\cap \mathfrak{q}^nM/ \mathfrak{q}^{n+1}M
	$$
	has support contained in $V(G_A(\mathfrak{q})_+)$. It is a finitely generated $A/\mathfrak{q}$-module 
	and $[0 :_{G_M(\mathfrak{q})} a^{\star}]_n = 0$ for $n\gg 0$. Because $a$ is $M$-regular the 
	Artin-Rees Lemma implies that there is an integer $\ell$ such that $\mathfrak{q}^{n+c}M :_M a = 
	\mathfrak{q}^n M$ for all $ n \geq \ell$. As a consequence we get 
	\[
	a M \cap \mathfrak{q}^{n} M = a \mathfrak{q}^{n-c} M \text{ for all } n \gg 0. 
	\]
	By inspecting the corresponding graded pieces there is a short exact sequence of graded 
	$R_A(\mathfrak{q})$-modules 
	\[
	0 \to \mathcal{M} \to R_M(\mathfrak{q})/(aT^c)R_M(\mathfrak{q}) \to R_{M/aM}(\mathfrak{q}) \to 0.
	\]
	Since $\mathcal{M}$ is a finitely generated $A/\mathfrak{q}$-module, it is of $\underline{aT^c}$-torsion 
	and the long exact cohomology sequence provides the short exact sequence as well as the 
	isomorphisms of the statement.
\end{proof}

Next we want to prove that our cohomologies allow to mod out the annihilator of a finitely generated 
module. To be more precise.

\begin{lemma} \label{seq-3}
	Let $\au, \mathfrak{q}, M$ as before with $I = \Ann_A M$and $\tilde{A} = A/I$. Then there are isomorphisms
	\begin{gather*}
	\check{H}^i(\au,\mathfrak{q},M;n) \cong \check{H}^i(\tilde{\au},\mathfrak{q}\tilde{A},M;n) \text{ and }\\
	\check{L}^i(\au,\mathfrak{q},M;n) \cong \check{L}^i(\tilde{\au},\mathfrak{q}\tilde{A},M;n)
	\end{gather*}
	for all $i,n \in \mathbb{N}$, where $\tilde{\au}$ denotes the sequence formed by $\au$ in $\tilde{A}$.
\end{lemma}

\begin{proof}
	First of all note that $a_i + I \in \mathfrak{q}^{c_i} +I$ for $i = 1,\ldots,t$ as elements in $\tilde{A}$.
	By base change of local cohomology we have isomorphisms
	\[
	H^i_{\underline{aT^c}}(R_M(\mathfrak{q})) \cong H^i_{\underline{\tilde{a}T^c}}(R_M(\mathfrak{q}\tilde{A})) \text{ and } 
	H^i_{\au A}(M) \cong H^i_{\tilde{\au}\tilde{A}}(M)
	\] 
	for all $i \in \mathbb{N}$. Because of 
	$$
	H^i_{\underline{aT^c}}(R_M(\mathfrak{q}))_n \cong \check{H}^i(\au,\mathfrak{q},M;n) \text{ and } 
    H^i_{\underline{\tilde{a}T^c}}(R_M(\mathfrak{q}\tilde{A}))_n \cong H^i(\tilde{\au},\mathfrak{q}\tilde{A},M;n)
    $$ 
    this proves the first family of isomorphisms. The second one follows by view of the exact sequence at the end 
    of the proof of \ref{cech-3}.
\end{proof}

We prove now the first structural result on a certain cohomology. 

\begin{theorem} \label{seq-4}
	With the previous notation suppose that $(A,\mathfrak{m})$ is a local ring and $\au A$ is an 
	$\mathfrak{m}$-primary ideal. Then 
	\[
	\check{H}^i(\au,\mathfrak{q},M;n) \; \text{ and } \; \check{L}^i(\au,\mathfrak{q},M;n)
	\]
	are Artinian $A$-modules for all $i,n \in \mathbb{N}$.
\end{theorem}

\begin{proof}
	We start with the proof for $\check{H}^i(\au,\mathfrak{q},M;n) \cong  H^{i}_{\underline{aT^c}}(R_M(\mathfrak{q}))_n$. 
	We show that $H^{i}_{\underline{aT^c}}(R_M(\mathfrak{q}))_n$ is  Artinian for all finitely 
	generated $A$-modules $M$ and all $n \in \mathbb{N}$ by an induction on $i$. 
	
	For $i= 0$ we have that $H^{0}_{\underline{aT^c}}(R_M(\mathfrak{q})) \cong \mathcal{N}$ with $\mathcal{N}_n 
	\cong (0 :_M \langle \au \rangle) \cap \mathfrak{q}^n M$. As an $A$-module of finite length $\mathcal{N}_n$ 
	is an Artinian $A$-module for all $n \in \mathbb{N}$. 
	
	So let $i > 0$. By view of \ref{seq-1} we may assume that $\mathcal{N} = 0$. 
	Therefore we may choose an $M$-regular element $a$ satisfying the requirements of the element $a$ in 
	\ref{seq-2}. This can be done by prime avoidance since $\Ass G_M(\mathfrak{q}) \not\subseteq V(G_A(\mathfrak{q})_+)$ 
	because otherwise $\dim G_M(\mathfrak{q}) = \dim M =0$. By induction hypothesis (see \ref{seq-2}) 
	it follows that $H^{i-1}_{\underline{aT^c}}(R_M(\mathfrak{q})/(aT^c)R_M(\mathfrak{q}))_n$ is Artinian for 
	all $n \in \mathbb{N}$. In order to prove that $H^i_{\underline{aT^c}}(R_M(\mathfrak{q}))_n$ is an 
	Artinian $A$-module we have to show 
	\begin{enumerate}
		\item $\Supp_A H^i_{\underline{aT^c}}(R_M(\mathfrak{q}))_n \subseteq \{\mathfrak{m}\}$ and 
		\item $\dim \Hom_A(\Bbbk,H^i_{\underline{aT^c}}(R_M(\mathfrak{q}))_n) < \infty$.
	\end{enumerate}
	The first result is easily seen since $\check{C}_{\underline{aT^c}}(R_M(\mathfrak{q}))_n \otimes_A A_{\mathfrak{p}}$ 
	is exact for any non-maximal prime ideal $\mathfrak{p}$. For the proof of (2) we use the 
	short exact sequence 
	$$
	0 \to R_M(\mathfrak{q})(-c) \stackrel{aT^c}{\to} R_M(\mathfrak{q}) \to R_M(\mathfrak{q})/(aT^c)R_M(\mathfrak{q}) \to 0.
	$$
	By applying local cohomology and restriction to degree $n$ it implies a surjection
	\[
	H^{i-1}_{\underline{aT^c}}( R_M(\mathfrak{q})/(aT^c)R_M(\mathfrak{q}))_n \to 
	\Hom_A(A/aA, [H^{i}_{\underline{aT^c}}(R_M(\mathfrak{q}))]_{n-c}) \to 0.
	\]
	That is, $\Hom_A(A/aA, [H^{i}_{\underline{aT^c}}(R_M(\mathfrak{q}))]_{n-c})$ is an Artinian $A$-module and 
	by adjointness 
	\[
	\dim \Hom_A(\Bbbk, [H^{i}_{\underline{aT^c}}(R_M(\mathfrak{q}))]_{n-c}) 
	= \dim \Hom_A(\Bbbk,\Hom_A(A/aA, [H^{i}_{\underline{aT^c}}(R_M(\mathfrak{q}))]_{n-c} )) 
	< \infty.
	\]
	Therefore $ [H^{i}_{\underline{aT^c}}(R_M(\mathfrak{q}))]_n$ is an Artinian $A$-module for all $n \in \mathbb{N}$. 
	This completes the inductive step.

	For the Artinianess of $\check{L}^i(\au,\mathfrak{q},M;n)$ we first note that the local cohomology modules 
	$H^i_{\au A}(M)$ are Artinian $A$-modules. Then the claim follows by virtue of the long exact sequence 
	\[
	\ldots \to  \check{H}^i(\au,\mathfrak{q},M;n) \to H^i_{\au A}(M) \to \check{L}^i(\au,\mathfrak{q},M;n) \to \ldots
	\]
	as follows from the short exact sequence of the corresponding \v{C}ech 
	complexes.
\end{proof}

\section{Vanishing results}
Let $\au = a_1,\ldots,a_t$ denote a system of elements of $A$. Let $\mathfrak{q} \subset A$ be an ideal and 
let $M$ be a finitely generated $A$-module. Let $a_i \in \mathfrak{q}^{c_i}, i = 1,\ldots,t$. 
Then we have the following technical lemma. 

\begin{lemma} \label{van-0}
	With the previous notation let $a \in \mathfrak{q}^c$. Then the following conditions 
	are equivalent:
	\begin{itemize}
		\item[(i)] $a^{\star} \notin \mathfrak{P} \text{ for all } 
		\mathfrak{P} \in \Ass G_{M}(\mathfrak{q}) \setminus V(G_A(\mathfrak{q})_+)$.
		\item[(ii)] There are integers $l,k$ such that $(\mathfrak{q}^{n+c}M :_Ma) \cap \mathfrak{q}^lM =
		\mathfrak{q}^n M$ for all $n \geq k$.
	\end{itemize}
\end{lemma}

\begin{proof}
	First note that $0:_{G_M(\mathfrak{q})} a^{\star} = \oplus_{n \geq 0} (\mathfrak{q}^{n+1+c}M :_M a) \cap 
	\mathfrak{q}^nM /\mathfrak{q}^{n+1}M$. Its support is contained in $V(G_M(\mathfrak{q})_+)$ 
	if and only if it is a finitely generated $G_A(\mathfrak{q})/G_A(\mathfrak{q})_+ = A/\mathfrak{q}$-module. 
	This is equivalent to 
	$$
	[0:_{G_M(\mathfrak{q})} a^{\star}]_n = (\mathfrak{q}^{n+1+c}M :_M a) \cap 
	\mathfrak{q}^nM /\mathfrak{q}^{n+1}M = 0
	$$
	for all $n \gg 0$. Now the condition in (i) is equivalent to the fact that the support of 
	$0:_{G_M(\mathfrak{q})} a^{\star}$ is contained in $V(G_M(\mathfrak{q})_+)$. Whence the equivalence 
	of (i) and (ii) follows easily. 
\end{proof}

Note that in case of $c=1$ in the above statement the element $a \in A$ is called a superficial element 
of $\mathfrak{q}$ with respect to $M$ (see \cite[Section 8.5]{SH} for more information). 

\begin{corollary} \label{van-1}
	With the notation of \ref{van-0} suppose that $a^{\star} \notin \mathfrak{P} \text{ for all } 
	\mathfrak{P} \in \Ass G_{M}(\mathfrak{q}) \setminus V(G_A(\mathfrak{q})_+)$. Then 
	$a^{\star} \notin \mathfrak{P} \text{ for all } 
	\mathfrak{P} \in \Ass G_{M/0:_M a}(\mathfrak{q}) \setminus V(G_A(\mathfrak{q})_+)$
\end{corollary}

\begin{proof}
	By virtue of the equivalent conditions in \ref{van-0} it will be enough to show that 
	\[
	(\mathfrak{q}^{n+2c}M, 0:_M a) :_M a\cap (\mathfrak{q}^l M, 0:_M a) = (\mathfrak{q}^n M, 0:_Ma) 
	\text{ for all } n \geq k.
	\]
	Let $m$ be an element of the left side. Then $am \in \mathfrak{q}^{n+2c}M :_Ma \cap \mathfrak{q}^lM = 
	\mathfrak{q}^{n+c}M$ as follows by the assumption. Therefore 
	$$
	m \in (\mathfrak{q}^{n+c}M :_M a) \cap (\mathfrak{q}^lM, 0:_M a) = 
	(\mathfrak{q}^{n+c} M:_M a \cap \mathfrak{q}^l M, 0:_M a) = (\mathfrak{q}^n M, 0:_M a)
	$$
	for all  $n \geq k$. Since the opposite inclusion is trivial the result follows by \ref{van-0}. 
\end{proof}

In the next we investigate the following condition for a system of elements $\au = a_1,\ldots,a_t$ 
with $a_i \in \mathfrak{q}^{c_i}, i = 1,\ldots,t$, namely 
\[
(\star) \hspace{.5cm} a_i^{\star} \notin \mathfrak{P} \text{ for all } 
\mathfrak{P} \in \Ass G_{M/(a_1,\ldots,a_{i-1})M}(\mathfrak{q}) 
\setminus V(G_A(\mathfrak{q})_+) \text{ for } i = 1,\ldots,t.
\]
In the case of $c_1 = \ldots= c_t = 1$ the condition ($\star$) is that of a superficial sequence 
of $\mathfrak{q}$ with respect to $M$ (see \cite{SH}).

\begin{theorem} \label{van-2}
	We fix the previous notation. Suppose that $\au = a_1,\ldots,a_t$ is a system of parameters satisfies the condition ($\star$). Then  $$\check{H}^t(\underline{a},\mathfrak{q},M;n) = H^t_{\underline{aT^c}}(R_M(\mathfrak{q}))_n = 0$$ 
	for all $n \gg 0$.
\end{theorem}

\begin{proof}
	First consider the following short exact sequence 
	\[
	0 \to \mathcal{N} 
	\to R_M(\mathfrak{q}) \to 
	R_{M/0:_M \langle a_1\rangle}(\mathfrak{q}) \to 0,
	\]
	where $\mathcal{N} = 0:_{R_M(\mathfrak{q})}  \langle a_1T^{c_1} \rangle$. 
	Recall that $[0:_{R_M(\mathfrak{q})} \langle a_1T^{c_1}\rangle]_n = (0:_M \langle a_1\rangle) \cap \mathfrak{q}^n M$ 
	for all $n \geq 0$. By base change on the local cohomology it follows that 
	$H^t_{\underline{aT^c}}(\mathcal{N}) = 0$. Note that $\mathcal{N}$ is annihilated 
	by a power of $(a_1 T^{c_1})$. By the long exact cohomology sequence it follows that 
	$$
	H^t_{\underline{aT^c}}(R_M(\mathfrak{q})) \cong 
	H^t_{\underline{aT^c}}(R_{M/0:_M \langle a_1 \rangle}(\mathfrak{q})).
	$$ 
	Therefore, we may assume that $a_1$ is $M$-regular as follows by an iterated use of \ref{van-1} 
	because of $0:_M \langle a_1 \rangle = 0:_M a^l_1$ for a certain large integer $l$. 

    If $t = 1$, then $\dim_A M = 1$ and $M/a_1M$ is of finite length. Then $[R_M(\mathfrak{q})/(a_1T^{c_1})R_M(\mathfrak{q})]_n =\mathfrak{q}^nM/a_1\mathfrak{q}^{n-c_1}M$. 
    By the assumption and since $a \in A$ is $M$-regular it follows that 
    $a_1M \cap \mathfrak{q}^n M = a_1 \mathfrak{q}^{n-c_1}M$ for all $n \gg 0$. Since $M/a_1M$ is of 
    finite length we get that $\mathfrak{q}^n M \subset a_1M$ for all $n \gg 0$, so that $\mathfrak{q}^n M = a_1\mathfrak{q}^{n-c_1}M$ 
    for all $n \gg 0$. Therefore $[R_M(\mathfrak{q})/(a_1T^{c_1})R_M(\mathfrak{q})]_n = 0$ 
    for all $n \gg 0$.  
    
	Now we proceed by induction on $t$. If $t = 1$ the short exact sequence in the statement 
	of 	\ref{seq-2} yields $[H^{t-1}_{\underline{aT^{c}}}(R_M(\mathfrak{q})/(a_1T^{c_1})R_M(\mathfrak{q}))]_n = 0$ 
	for all $n \gg 0$. By induction hypothesis this holds for $t > 1$ by view of the 
	isomorphism in the statement in \ref{seq-2}. Now we show $[H^t_{\underline{aT^c}}(R_M(\mathfrak{q}))]_n = 0$
	for $n \gg 0$.
	To this end recall the short exact sequence 
	$$
	0 \to R_M(\mathfrak{q})(-c_1) \stackrel{a_1T^{c_1}}{\to} R_M(\mathfrak{q}) \to R_M(\mathfrak{q})/(aT^c)R_M(\mathfrak{q}) \to 0.
	$$
	By the vanishing of $[H^{t-1}_{\underline{aT^{c}}}(R_M(\mathfrak{q})/(a_1T^{c_1})R_M(\mathfrak{q}))]_n$ 
	for $n \gg 0$ it implies the bijectivity 
	\[
	H^t_{\underline{aT^c}}(R_M(\mathfrak{q}))(-c_1) \stackrel{a_1T^{c_1}}{\longrightarrow} H^t_{\underline{aT^c}}(R_M(\mathfrak{q}))
	\]
	in large degrees. Assume now there is an $0 \not= r \in [H^t_{\underline{aT^c}}(R_M(\mathfrak{q})]_n$ 
	for a large $n$. Because $\Supp H^t_{\underline{aT^c}}(R_M(\mathfrak{q})) \subseteq V(\underline{aT^c})$ 
	there is an integer $m$ such that $(a_1 T^{c_1})^m \cdot r = 0$. By the bijectivity this implies that $r =0$, 
	a contradiction. This completes the inductive step. 
\end{proof}

As a consequence of the previous result we have the following corollary.

\begin{corollary} \label{van-3} 
	We fix the previous notation. Suppose that $\au = a_1,\ldots,a_t$ is a system of parameters satisfies the condition ($\star$). Then 
	\[
	H^t_{\au}(M) \cong \check{L}^t(\au,\mathfrak{q},M;n)
	\]
	for all $n \gg 0$. In particular, if in addition $H^t_{\au}(M) \not= 0$, then 
	$\check{L}^t(\au,\mathfrak{q},M;n) \not= 0$ for all $n \gg 0$. 
\end{corollary}

\begin{proof}
	The short exact sequence of complexes as shown in \ref{cech-2} provides an exact sequence 
	\[
	H^t_{\underline{aT^c}}(R_M(\mathfrak{q}))_n \to H^t_{\au}(M) \to \check{L}^t(\au,\mathfrak{q},M;n) \to 0.
	\]	
	By virtue of Theorem \ref{van-2} the claim follows. 
\end{proof}

\section{Non Vanishing}

In this section let $\au = a_1,\ldots,a_t$ be a system of elements of a local ring $(A,\mathfrak{m},\Bbbk)$ 
that generates an $\mathfrak{m}$-primary ideal. Then $\mathfrak{q}$ is also an $\mathfrak{m}$-primary ideal. 
Let $M$ denote a finitely generated $A$-module. By view of \ref{seq-3} we may assume that $\dim A = \dim M$. By view 
of \ref{cech-3} we may assume that $\au$ is a system of parameters of $M$. Then it follows 
that $\check{L}^i(\au,\mathfrak{q},M;n) = 0$ for $i > \dim M$ and all $n \in \mathbb{Z}$. In this 
section we will investigate $\check{L}^t(\au,\mathfrak{q},M;n)$ for $t = \dim M$. 

\begin{definition} \label{nonv-1}
	Let $a \in \mathfrak{q}^c$. For an integer $n \geq 0$ it follows 
	$$
	\mathfrak{q}^n (\mathfrak{q}^c)^l M :_M a^l \subseteq 	\mathfrak{q}^n (\mathfrak{q}^c)^{l+1} M :_M a^{l+1}
	$$ 
	for all $l \geq 0$. We denote the stable value by $(\mathfrak{q}^nM)^{aA}$. Then $\mathfrak{q}^n M 
	\subseteq (\mathfrak{q}^n M)^{aA}$ and $0:_M \langle a \rangle 
	\subseteq (\mathfrak{q}^nM)^{aA}$
	for all $n \geq 1$. Moreover it is easily seen that $\mathfrak{q}^n M[\mathfrak{q}^c/a] \cap M = 
	(\mathfrak{q}^n M)^{aA}$ for all $n \geq 1$.
\end{definition}

Then we have the following result.

\begin{proposition} \label{nonv-2}
	With the previous notation the following conditions are equivalent:
	\begin{itemize}
		\item[(i)] $(\mathfrak{q}M)^{aA} =M$.
		\item[(ii)] $(\mathfrak{q}^k M)^{aA} = M$ for all $k \geq 1$.
		\item[(iii)] $(\mathfrak{q}^k M)^{aA} = M$ for some $ k \geq 1$.
	\end{itemize}
\end{proposition}

\begin{proof}
	If (i) is satisfied, then $a^lM \subseteq \mathfrak{q}^{1+cl}M$ for an integer $l \geq 1$. This implies 
	$a^{kl}M \subseteq \mathfrak{q}^k (\mathfrak{q}^c)^{kl}M$ for al $k \geq 1$. This yields 
	$M \subseteq \mathfrak{q}^k (\mathfrak{q}^c)^{kl}M :_M a^{kl} \subseteq (\mathfrak{q}^kM)^{aA}$ which 
	proves (ii). Since (ii) $\Longrightarrow$ (iii) is trivially true we have to show (iii) $\Longrightarrow$ (i). 
	Let $m \in M$ denote an arbitrary element. Then 
	$$
	a^l m \in \mathfrak{q}^k (\mathfrak{q}^{c})^l M \subseteq \mathfrak{q} (\mathfrak{q}^{c})^l M
	$$ 
	for a certain $ l \in \mathbb{N}$. That is $m \in (\mathfrak{q}M)^{aA}$ and the proof is finished. 
\end{proof}

\begin{lemma} \label{nonv-3}
	With the previous notation the following conditions are equivalent:
	\begin{itemize}
		\item[(i)] $\check{L}^1(a,\mathfrak{q},M;1) =0$.
		\item[(ii)] $\check{L}^1(a,\mathfrak{q},M;n) =0$ for all $n \geq 1$.
		\item[(iii)] $\check{L}^1(a,\mathfrak{q},M;n) =0$ for some $n \geq 1$.
		\item[(iv)] $(\mathfrak{q}M)^{aA} =M$.
	\end{itemize}
\end{lemma}

\begin{proof}
	By the definition we have that $\check{L}^1(a,\mathfrak{q},M;n) \cong \varinjlim \{M/(a^k,\mathfrak{q}^{n+kc})M,a\}$, where the maps in the direct system are given by multiplication by $a$. 
	Then the direct limit is zero if and only if for each $k \geq 1$ there is an integer $l$ such that 
	$a^lM \subseteq (a^{k+l}M, \mathfrak{q}^{n+(k+l)c}M)$. This holds if and only if 
	$M \subseteq (a^{k}M, \mathfrak{q}^{n+kc} (\mathfrak{q}^c)^lM :_M a^l)$ and by Nakayama Lemma if and only if 
	$M = \mathfrak{q}^{n+kc} (\mathfrak{q}^c)^lM :_M a^l$. But this is equivalent to 
	$M = (\mathfrak{q}^{n+kc}M)^{aA}$. Then the equivalence of all the conditions of the statement 
	is a consequence of \ref{nonv-2}.
\end{proof}

\begin{remark} \label{nonv-4}
	With the previous notation it is not clear to us when $(\mathfrak{q}M)^{aA}$ is a proper submodule 
	of $M$. Suppose that $a^\star \notin \mathfrak{P}$ for all $\mathfrak{P} \in \Ass G_M(\mathfrak{q}) 
	\setminus V(G_A(\mathfrak{q})_+)$ and that $a$ is $M$-regular. Then $\mathfrak{q}^{n+c}M :_M a = 
	\mathfrak{q}^nM$ for all $n \gg 0$. This implies that $\mathfrak{q}^n (\mathfrak{q}^c)^l M:_M a^l = 
	\mathfrak{q}^n M$ for all $n \gg 0$. That is $(\mathfrak{q}^nM)^{aA} = \mathfrak{q}^n M$ for all $n \gg 0$ 
	and $\check{L}^1(a,\mathfrak{q},M;n) \not=0$ for all $n \geq 1$ (see \ref{nonv-2}). 
	Moreover, if $(\mathfrak{q}^nM)^{aA} = \mathfrak{q}^n M$ for all $n \gg 0$ then it follows also 
	$0:_M a = 0$ and $\mathfrak{q}^{n+c}M :_M a = \mathfrak{q}^n M$ for all $n \gg 0$ as easily seen. 
\end{remark}

Now we are ready to prove the non-vanishing result.

\begin{theorem} \label{nonv-5}
	With the previous notation let $\au = a_1,\ldots,a_t \subset \mathfrak{q}$ denote a system of parameters 
	of $M$. Then $\check{L}^t(\au,\mathfrak{q},M;n) \not= 0$ for all $n \gg 0$.
\end{theorem}

\begin{proof}
	By view of \ref{seq-3} we may assume that $\Ann M = 0$. Then $\dim G_M(\mathfrak{q}) = t$. Inductively 
	we choose a system of elements $\underline{b} = b_1,\ldots,b_t$ such that $\underline{b}$ satisfies 
	the condition ($\star$) above. Then $\underline{b}$ is a system of parameters of $M$ and 
	$\check{L}^t(\au,\mathfrak{q},M;n) \cong \check{L}^t(\underline{b},\mathfrak{q},M;n)$ (see \ref{cech-3}) .
	Then by \ref{van-3} it follows that $0 \not= H^t_{\underline{b}}(M) \cong \check{L}^t(\au,\mathfrak{q},M;n)$ 
	for all $n \gg 0$. Note that the non-vanishing of $H^t_{\underline{b}}(M)$ follows by virtue of 
	\cite{aG} since $\dim M = t$ and $\underline{b}$ generates an $\mathfrak{m}$-primary ideal. 
\end{proof}

\end{document}